\documentclass[11pt]{amsart}
\usepackage{amsmath,amssymb,amsfonts,amsthm, amscd,indentfirst}
\usepackage{amsmath,latexsym,amssymb,amsmath,
amscd,amsthm,amsxtra}
\usepackage{hyperref}
\usepackage{color}

\newtheorem{theorem}{Theorem}[section]
\newtheorem{prop}{Proposition}[section]

\newtheorem{definition}{Definition}[section]
\newtheorem{corollary}{Corollary}[section]

\newtheorem{remark}{\textbf{Remark}}[section]

\def\rr{\mathbb{R}}
\def\ss{\mathbb{S}}

\def\bb{\mathbb{B}}

\def\O{\Omega}
\def\p{\partial}
\def\e{\epsilon}
\def\a{\alpha}

\def\g{\gamma}
\def\d{\delta}

\def\s{\sigma}

\def\p{\partial}

\def\S{{\Sigma}}
\def\<{\langle}
\def\>{\rangle}
\def\div{{\rm div}}
\def\n{\nabla}

\def \ds{\displaystyle}
\def \vs{\vspace*{0.1cm}}

\begin{document}

\title[Guan-Li type mean curvature flow]{Guan-Li type mean curvature flow for  free boundary  hypersurfaces in a ball}
\author{Guofang Wang}\thanks{GW is partly supported by  Priority Programme  ``Geometry at Infinity'' (SPP 2026) of DFG}
\address{ Albert-Ludwigs-Universit\"at Freiburg,
Mathematisches Institut,
Ernst-Zermelo-Str. 1,
79104 Freiburg, Germany}
\email{guofang.wang@math.uni-freiburg.de}
\author{Chao Xia}
\address{School of Mathematical Sciences\\
Xiamen University\\
361005, Xiamen, P.R. China}
\email{chaoxia@xmu.edu.cn}\thanks{CX is supported by NSFC (Grant No. 11871406), the Natural Science Foundation of Fujian Province of China (Grant No. 2017J06003) and the Fundamental Research Funds for the Central Universities (Grant No. 20720180009).}

\begin{abstract}
    In this paper we introduce a Guan-Li type volume preserving mean curvature flow for  free boundary hypersurfaces in a ball. We give a concept of star-shaped free boundary hypersurfaces in a ball and show that the Guan-Li type mean curvature flow has long time existence and converges to a free boundary spherical cap, provided the initial data is star-shaped.
     \end{abstract}

\subjclass[2010]{53C21, 53C44}

\date{}
\maketitle

\medskip

\section{Introduction}
 Let $\bb^{n+1}\subset\rr^{n+1}$ be the open unit Euclidean ball centered at the origin and   $\ss^n= \partial\bb^{n+1}\subset\rr^{n+1} $ the unit sphere.
 In this paper, we shall consider a mean curvature type flow for compact hypersurfaces in $\bb^{n+1}$ with free boundary on $\ss^n$.  
Let $\S\subset \bar \bb^{n+1}$ be a properly embedded compact hypersurface with boundary, which is given by 
$$x: M\to \bar \bb^{n+1},$$ where $M$ is a compact Riemannian manifold with boundary $\partial M$.
Here properly embedded means that $${\rm int}(\S)=x({\rm int}(M))\subset \bb^{n+1} \quad \hbox{ and } \quad \p\S=x(\p M)\subset \p \bb^{n+1}.$$
We further assume that $\S$ has  free boundary, in the sense that $\S$ intersects $\p \bb^{n+1}=\ss^n$ orthogonally, that is,
$$\<\nu, \mu\circ x\>=0 \quad \hbox{ on } \p M,$$
where $\nu$ is a unit normal vector field of $x$, which will be specified later, and  $\mu$ is the outward unit normal vector field of $\ss^n$,
i.e., $\mu\circ x=x$ along $\p M$.

Let $e\in \ss^n\subset \rr^{n+1}$ be a fixed unit vector field.
Consider a family of properly embedded compact hypersurfaces $\{\S_t\}_{t\in [0, T)}$ with  free boundary, 
given by embeddings $$x: M\times [0, T)\to \bar \bb^{n+1},$$  satisfying
\begin{equation}\label{flow}
\left\{
\begin{array}{rcll}
\p_t x& =&\ds \vs (n \<x, e\> -H\<X_e, \nu\>)\nu & \hbox{ in }M\times [0, T),\\
\<\nu, \mu\circ x\>&= &0  & \hbox{ on } \p M\times [0, T). 
\end{array}
\right.
\end{equation}
with an initial surface $x(\cdot, 0)=x_0.$
Here $\nu$ and $H$ are a unit normal vector field and the mean curvature of $x(\cdot, t)$ respectively,  $X_e$ is a fixed vector field in $\rr^{n+1}$ given by
$$X_e=X_e(x)=\<x, e\>x-\frac12(|x|^2+1)e,$$
for a fixed unit vector $e$.
This vector field plays an important role in our recent paper \cite{WX_2019}.
We choose $\nu$ in the following way. Let $\O_t$ be the component of the enclosed domain by $\S_t$ and $\ss^n$ which contains $e$ in its interior. 
Then $\nu$ is chosen to be the outward normal of $\S_t$ with respect to $\O_t$. 
Also, throughout this paper, we make the convention that  the enclosed domain $\O_t$ of $\S_t$ and $\ss^n$ is the one  $e$ in its interior.
The volume of the enclosed domain $\O_t$ of $\S$ is called the enclosed volume of $\S_t$.

The flow is designed in this way so that the enclosed volume of $\S_t$ is preserved along the flow \eqref{flow}. We will discuss it later. Such kinds of flow was first considered by Guan-Li \cite{GL} in the setting of closed hypersurfaces in space forms and by Guan-Li-Wang \cite{GLW}
in the setting of closed hypersurfaces in warped product spaces.

The main objective of this paper is to study the existence and  the convergence of the flow \eqref{flow}. For this aim we 
introduce a concept of {\it star-shaped hypersurfaces with free boundary in $\bar \bb^{n+1}$}. 
To arrive at this, we should first make some comments on the vector field $X_e$ above.
$X_{e}$ is a conformal Killing vector field with $$\<X_e(x), x\>=0, \forall x\in \p \bb^{n+1}.$$ More precisely,  denoting the Euclidean metric by $\delta$, we have $$\mathcal{L}_{X_e}\delta=\<x, e\>\delta.
$$  
Let  $\phi_t: \bar \bb^{n+1}\to \bar \bb^{n+1}$ be the  one-parameter family of conformal transformations
generated by $X_{e}$. 
Let $\Pi_e$ be the hyperplane which passes through the origin and is orthogonal to $e$. For each point $p\in \Pi_e$, there exists a unique planar circle  
passing through $p$ and $\pm e$. One can check that the integral curves of $X_e$ are given by the intersection of all such planar circles with $\bb^{n+1}$.
We introduce star-shaped hypersurfaces with free boundary in $\bar \bb^{n+1}$.
\begin{definition} 
 1). A proper embedded hypersurface $\S\subset \bar\bb^{n+1}$ is called 
             {\rm star-shaped  (with respect to $e$)} if $\S$ intersects each integral curve of $X_e$ exactly once. 
             
 2).
   A proper embedded hypersurface $\S\subset \bar\bb^{n+1}$ is called {\rm strictly star-shaped  (with respect to $e$)} if 
   \begin{equation}\label{star}
 \<X_e, \nu\>>0.
\end{equation}
\end{definition}

 For our purpose we will consider   strictly star-shaped hypersurfaces in $\bar \bb^{n+1}$ in this paper. 
 This condition is  slightly stronger than the condition of star-shapedness, but clearly much weaker than the convexity. 
For the simplicity in this paper we call hypersurfaces satisfying \eqref{star} star-shaped hypersurfaces.

From now on we consider star-shaped hypersurfaces. Being such a hypersurface, it is necessary that $M$ is of ball type. Therefore we use $M=  \bar \ss_+^n$,
the closed hemisphere.

Our main result is the following
\begin{theorem}\label{mainthm}
Let $\S\subset \bar \bb^{n+1} (n\ge 2)$  be a properly embedded compact hypersurface with free boundary,
given by $x_0: \bar \ss_+^n\to \bar \bb^{n+1},$ which is star-shaped with respect to $e$. Then there exists a unique solution 
$x: \bar \ss_+^n\times [0, \infty)\to \bar \bb^{n+1}$ to \eqref{flow}.
Moreover, $x(\cdot, t)$ converges smoothly to a spherical cap { or the totally geodesic $n$-ball}, whose enclosed domain has the same volume as $\S$. 
When $n\ge 3$, or $n=2$ and the enclosed volume of $x_0$ is not that of a half ball, $x(\cdot, t)$ converges exponentially fast.
\end{theorem}
{ The family of spherical caps is given by   $$C^{\pm}_{r}( e)=\{x\in \bar \bb^{n+1}: |x\pm \sqrt{r^2+1} e| = r\}, r>0$$ and the totally geodesic $n$-ball
is given by $$C_\infty(e)=\{x\in \bar \bb^{n+1}: \<x, e\>=0\}.$$} 
It is clear that either  each spherical cap $C^{\pm}_r(e)$ or  the totally geodesic $n$-ball $C_\infty(e)$ has free boundary, that is, 
it intersects the support ${\mathbb S}^n$ orthogonally.

As a direct consequence, we give a flow proof of the isoperimetric problem for free boundary hypersurfaces in $\bb^{n+1}$.
\begin{corollary}\label{coro1}
 Among  star-shaped free boundary hypersurfaces with fixed enclosed volume, the  totally geodesic $n$-ball or the spherical caps have minimal area. 
\end{corollary}
For general hypersurfaces it is a classical result  proved  by Burago-Mazaya  \cite{BM},
 Bokowsky-Sperner \cite{BS} and Almgren \cite{A}, by using the method of symmetrization.

The introduction of flow \eqref{flow} is motivated by the paper of Guan-Li \cite{GL}, in which they used at the first time 
the Minkowski formula for closed hypersurfaces to define a geometric flow for isoperimetric problems. In the same spirit, the flow \eqref{flow} is based on the following two Minkowski formulas obtained in \cite{WX_2019} for free boundary hypersurfaces
\begin{eqnarray}\label{eq3}
n\int_\Sigma \< x, e \> &=&  \int_\Sigma \< X_e, \nu\> H,\\
\label{eq4}
\int_\Sigma \<x, e\>  H&=&\frac 2{n-1}\int  \<X_e, \nu \> \sigma_2 (\kappa).
\end{eqnarray}
Here $\kappa=(\kappa_1, \kappa_2,\cdots, \kappa_n)$ are principal curvatures of $\Sigma$ and $\sigma_2(\kappa)$ is the 2nd order mean curvature.
From these formulas, one can show that flow \eqref{flow} preserves the volume of $\Omega_t$ and decreases the area of $\Sigma_t$. See Proposition \ref{pro3.3}. These are crucial properties of this flow.

To prove Theorem \ref{mainthm}, we first transform the flow equation to a scalar flow \eqref{flow2} on $\ss^n_+$ by using star-shapedness. By using the M\"obius transformation between the half space $\bar \rr^{n+1}_+$ and the unit ball $\bar \bb^{n+1}$, a star-shaped hypersurface in $\bar \bb^{n+1}$ is equivalent to a classical star-shaped hypersurface in $\bar \rr^{n+1}_+$ with a conformal flat metric. We remark that a different reparametrization based on M\"obius transformation between round cylinder and $\bb^{n+1}$ was used by Lambert-Scheuer \cite{LS}. For the scalar flow \eqref{flow2}, the $C^0$ estimate follows directly from the barrier argument. We then show the gradient estimate for \eqref{flow2}.

Finally we mention some previous results on curvature flows with free boundary in $\bb^{n+1}$. The classical mean curvature flow was considered by Stahl \cite{St1, St2}, where it was shown that strictly convex initial data are driven to a round point in a finite time. The classical inverse mean curvature flow was treated by Lambert-Scheuer \cite{LS}, where it was shown that strictly convex initial data are driven to a flat perpendicular $n$-ball in a finite time. 
Following a similar idea of this paper, a fully nonlinear inverse curvature type flow was considered by Scheuer and the authors \cite{SWX} to show a class of new Alexandrov-Fenchel's inequalities for convex free boundary hypersurfaces in $\bb^{n+1}$.

The rest of this paper is organized as follows. In Section 2 we introduce the M\"obius transformation between $\bar \rr^n_+$ and $\bar \bb^n$, and reduce flow \eqref{flow} to a scalar flow \eqref{flow2}, provided that all  evolving hypersurfaces are star-shaped. In Section 3, we show that $C^0$ and $C^1$ estimates of \eqref{flow}.
As consequence, we prove in Section 4 that the global convergence of \eqref{flow}, Theorem \ref{mainthm} and  its consequence, Corollary \ref{coro1}.

\

\section{A scalar flow}

In this section we reduce \eqref{flow} to a scalar flow, provided that all evolving hypersurfaces are star-shaped.

Without loss of generality, from now on, we assume $e=E_{n+1}$, the $(n+1)$-coordinate vector.
Let $$\rr^{n+1}_+=\{z=(z_1,\cdots, z_{n+1})\in \rr^{n+1}: z_{n+1}> 0\}$$ be the half space.
Define
\begin{eqnarray}
f: \bar \rr^{n+1}_+&\to&\bar \bb^{n+1},\\
(z', z_{n+1})&\mapsto&\left(\frac{2z'}{|z'|^2+(1+z_{n+1})^2},\quad \frac{|z|^2-1}{|z'|^2+(1+z_{n+1})^2}\right).
\end{eqnarray}
Here $z'=(z_1,\cdots, z_{n})\in \mathbb{R}^{n}$. $f$ is bijective and \begin{eqnarray}
&&f(\rr^{n+1}_+)=\bb^{n+1},\\&&f(\p \rr^{n+1}_+)=\p \bb^{n+1}, \\&&f(\{|z|=1\})=\{x_{n+1}=0\}.
\end{eqnarray}
Moreover, $f$ is a conformal diffeomorphism between $(\bar \rr^{n+1}_+, \delta_{\bar \rr^{n+1}_+})$ and $(\bar \bb^{n+1}, \delta_{\bar \bb})$. Here $\delta_{\bar \rr^{n+1}_+}$ and $\delta_{\bar \bb}$ denote the restriction of the Euclidean metric to $\bar \rr^{n+1}_+$ and $\bar \bb^{n+1}$ respectively. Precisely, $$f^*\delta_{\bar \bb}= e^{2w}\delta_{\bar \rr^{n+1}_+}= \frac{4}{(|z'|^2+(1+z_{n+1})^2)^2}\delta_{\bar \rr^{n+1}_+}.$$
In other words, $(\bar \bb^{n+1}, \delta_{\bar \bb})$ and $(\bar \rr^{n+1}_+, e^{2w}\delta_{\bar \rr^{n+1}_+})$ are isometric.

In $\bar \rr^{n+1}_+$, we use the polar coordinates $(\rho,\varphi, \theta)\in [0,\infty)\times [0,\frac{\pi}{2}]\times \ss^{n-1}$, where  $$\rho^2=|z'|^2+z_{n+1}^2, \quad z_{n+1}=\rho\cos\varphi$$ and $\theta\in  \ss^{n-1}$ is the spherical coordinate.

By using  $(\rho,\varphi, \theta)$ in $\bar \rr^{n+1}_+$,  the mapping $f$ can be rewritten as
\begin{eqnarray}\label{conf-polar}
f(\rho,\varphi, \theta)=\left(\frac{2\rho\sin \varphi \vec{\theta}}{1+\rho^2+2\rho\cos \varphi}, \frac{\rho^2-1}{1+\rho^2+2\rho\cos\varphi}\right).
\end{eqnarray}
Here $ \vec{\theta}$ denotes the position vector of the point $\frac{z'}{|z'|}\in \ss^{n-1}$.
We also have
$$f^*\delta_{\bar \bb}=e^{2w} \delta_{\bar \rr^{n+1}_+}=\frac{4}{(1+\rho^2+2\rho\cos\varphi)^2}(d\rho^2+\rho^2d\varphi^2+\rho^2\sin^2\varphi g_{\ss^{n-1}}),$$
where $$w=w(\rho, \varphi, \theta)=\log 2-\log (1+\rho^2+2\rho\cos\varphi).$$ 
One may also check that the conformal Killing vector field $X_{n+1}$ on $\bar \bb_+$ is transformed to \begin{eqnarray}\label{tildeX}
\tilde{X}=(f^{-1})_*(X_{n+1})=-\rho\p_\rho\hbox{ on } \bar \rr^{n+1}_+.
\end{eqnarray}
The integral curves of $\tilde{X}$ are clearly the rays  in $\rr^{n+1}_+$ initiating from the origin.

\

Let $\S\subset  \bar \bb^{n+1}$ be a properly embedded compact hypersurface with boundary, given by an embedding $x:\bar \ss^n_+\to \bar \bb^{n+1}$.  
We associate $\S$ with a corresponding hypersurface $\tilde{\S}\subset \bar \rr^{n+1}_+$ given 
by the embedding
$$\tilde{x}= f^{-1}\circ x: \bar \ss^n_+\to \bar \rr^{n+1}_+.$$
In view of \eqref{tildeX}, $\S$ is star-shaped with respect to $E_{n+1}$ if and only if $\tilde{\S}$ is star-shaped (with respect to the origin) in $\bar \rr^{n+1}_+$, that is, $\tilde{\S}$ intersects each of the rays in $\rr^{n+1}_+$ initiating from the origin exactly once, or in other words, $\tilde{\S}$ is a graph over $\bar \ss^n_+$.

Since $(\bar \bb^{n+1}, \delta_{\bar \bb})$ and $(\bar \rr^{n+1}_+, e^{2w}\delta_{\bar \rr^{n+1}_+})$ are isometric, 
a proper embedding $x:\bar \ss^n_+\to \bar \bb^{n+1}$ can be identified as an embedding $\tilde{x}: \bar \ss^n_+\to (\bar \rr^{n+1}_+, e^{2w}\delta_{\bar \rr^{n+1}_+}).$
In the following, we use $\tilde{}$ to indicate the corresponding quantity for $\tilde{x}: \bar \ss^n_+\to(\bar \rr^{n+1}_+, e^{2w}\delta_{\bar \rr^{n+1}_+})$.

Given a star-shaped hyersurface $\tilde{\S}$ in $(\bar \rr^{n+1}_+, e^{2w}\delta_{\bar \rr^{n+1}_+})$, by using the polar coordinate $(\rho, \varphi, \theta)\in \bar \rr^{n+1}_+$, 
we may write $$\tilde{x}=\rho(y)y=\rho(\varphi, \theta) y, y= (\varphi, \theta)\in \bar \ss^n_+.$$ 
We use $\s=d\varphi^2+\sin^2\varphi d\theta^2$ and $\nabla^\s$ to denote the round metric and the covariant derivative on $\bar \ss^n_+$. Set $$\g=\log \rho, \hbox{ and }v=\sqrt{1+|\n^\s \g|^2}.$$

We have the following correspondence for several geometric quantities. 
\begin{prop}\label{corresp}
\begin{itemize}\
\item[(i)]$$x_{n+1}=\<f(\tilde{x}), E_{n+1}\>=\frac12(\rho^2-1)e^{w}.$$
\item[(ii)]$$|X_{n+1}|=e^w|-\rho\p_\rho|=\rho e^w.$$
\item[(iii)]$$\<X_{n+1},\nu\>=e^{2w}\<-\rho\p_\rho,\tilde{\nu}\>=\frac{\rho e^w}{v}.$$
\item[(iv)]The Weingarten transformation $h_i^j= g^{jk}h_{ik}$ satisfies
\begin{eqnarray*}
h_i^j=\tilde{h}_i^j
&=&\frac{1}{\rho v e^w}(\s^{kj}-\frac{\g^k\g^j}{v^2})\g_{ik}+\left[\frac{\sin\varphi\g_\varphi}{v}+\frac{(\rho^2-1)}{2\rho v}\right]\delta_i^j.
\end{eqnarray*}
\item[(v)]\begin{eqnarray*}
H=\tilde{H}
&=&\frac{1}{\rho v e^w}(\s^{ij}-\frac{\g^i\g^j}{v^2})\g_{ij}+\frac{n\sin\varphi\g_\varphi}{v}+\frac{n(\rho^2-1)}{2\rho v}.
\end{eqnarray*}\end{itemize}
\end{prop}
\begin{remark}We see from (iii) that in case we have $C^0$ estimate, a positive lower bound for $\<X_{n+1},\nu\>$ is equivalent to the gradient estimate for $\g$.
\end{remark}
\begin{proof}(i) follows from  \eqref{conf-polar} and (ii) follows from \eqref{tildeX}.

It is clear that the unit outward normal is given by \begin{eqnarray}\label{normal}
\tilde{\nu}=e^{-w}\nu_\delta=e^{-w}\frac{\rho^{-1}\n^\s \g-\p_\rho}{v},
\end{eqnarray}
where $\nu_\delta$ is the unit outward normal of $\tilde{\S}\subset (\bar \rr^{n+1}_+, \delta_{\bar \rr^{n+1}_+})$.
Then (iii) follows from \eqref{tildeX} and \eqref{normal}.

By a well-known transformation law for the Weingarten transformation under a conformal change,  we know that $\tilde{h}_i^j$ of $\S\subset (\bar \rr^{n+1}_+, e^{2w}\delta_{\bar \rr^{n+1}_+})$ with respect to $-\tilde{\nu}$ is given by
\begin{eqnarray}\label{mean-curv}
\tilde{h}_i^j&=&e^{-w}((h_\delta)_{i}^j+\nabla^\d_{\nu_\d}w \delta_i^j),
\end{eqnarray}
where $(h_\delta)_{i}^j$ is the Weingarten transformation with respect to $-\nu_\d$ of $\tilde{\S}\subset (\bar \rr^{n+1}_+, \delta_{\bar \rr^{n+1}_+})$ and  $\nabla^\d$ is the Euclidean derivative.

It is known that \begin{eqnarray}\label{mean-curv-E}
(h_\delta)_{i}^j=-\frac{1}{\rho v}\delta_i^j+\frac{1}{\rho v}(\s^{kj}-\frac{\g^k\g^j}{v^2})\g_{ik},
\end{eqnarray}
On the other hand, using $e^{-w}=\frac12(1+\rho^2+2\rho\cos\varphi)$, we have
\begin{eqnarray}\label{normal-der}
\nabla^\d_{\nu_\d}(e^{-w})&=&\left\<(\rho+\cos \varphi)\p_\rho-\rho^{-1} \sin\varphi\p_\varphi,\frac{\rho^{-1}\n^\s \g-\p_\rho}{v}\right\>
\\&=& -\frac{1}{v}(\rho+\cos\varphi+\sin\varphi\g_\varphi).\nonumber
\end{eqnarray}
(iv) follows from \eqref{mean-curv}, \eqref{mean-curv-E}  and \eqref{normal-der}. (v) follows from (iv) by taking trace.
\end{proof}

We return to the flow problem \eqref{flow} in $(\bar \bb^{n+1}, \delta_{\bar \bb})$. By the identification using $f$,  the corresponding family of embeddings $\tilde{x}: \ss^n_+\to (\bar \rr^{n+1}_+, e^{2w}\delta_{\bar \rr^{n+1}_+})$ satisfies
\begin{equation}\label{flow1}
\left\{
\begin{array}{rcll}
\p_t \tilde{x}&=&(n \<f(\tilde{x}), E_{n+1}\> - \tilde{H} e^{2w}\<-\rho\p_\rho, \tilde{\nu}\>)\tilde{\nu}&\hbox{ in }\ss_+^n\times [0, T),\\
\<\tilde{\nu}, \tilde{\mu}\circ \tilde{x}\>&=&0,&\hbox{ on } \p \ss_+^n\times [0, T),
\end{array}
\right.
\end{equation}
with an initial surface $\tilde{x}(\cdot, 0)=\tilde{x}_0$. 
Here $\tilde{\mu}$ is the downward unit normal of $(\bar \rr^{n+1}_+, e^{2w}\delta_{\bar \rr^{n+1}_+})$.
As long as $\tilde{x}(\cdot, t)$ is star-shaped in $\bar \rr^{n+1}_+$, we may reduce \eqref{flow1} to a scalar flow.

Using a standard argument (see \cite{Ge}, Eq. (2.4.21)) and Proposition \ref{corresp}, we see that 
\begin{eqnarray}\label{flow-scalar}
\p_t \g&=& -\frac{v}{\rho e^w} \left(\frac{n}{2}(\rho^2-1)e^{w}- \tilde{H} \frac{\rho e^w}{v}\right)\\
&=&\frac{1}{\rho v e^w}\left(\s^{ij}-\frac{\g^i\g^j}{v^2}\right)\g_{ij}+\frac{n\sin\varphi\g_\varphi}{v}-\frac{n(\rho^2-1)|\n^\s \g|^2}{2\rho v}\nonumber
\\&=&\div_{\s}\left(\frac{\n^\s \g}{\rho v e^w}\right)-\frac{n+1}{v}\s\left(\n^\s \g, \n^\s\left(\frac{1}{\rho e^w}\right)\right).\nonumber\end{eqnarray}
The last line above follows from the fact 
\begin{eqnarray*}
\s\left(\n^\s \g, \n^\s\left(\frac{1}{\rho e^w}\right)\right)=\frac{\rho^2-1}{2\rho}|\n^\s \g|^2-\sin\varphi\g_\varphi.
\end{eqnarray*}

Next we examine the boundary condition. Note  that $\mu\perp \p \bb^{n+1}$. 
Since the conformal change $f$ preserves angles, we have $\tilde{\mu}\perp \p \rr^{n+1}_+$ and in turn $$\tilde{\mu}=-e^{-w}\p_\varphi.$$
In view of \eqref{normal}, the boundary condition in \eqref{flow1} reduces to
\begin{eqnarray}
\nabla^\s_{\p_\varphi}\g= 0 \hbox{ on }\p \ss^n_+.
\end{eqnarray}

In summary, the flow problem \eqref{flow1} reduces to solve the scalar PDE  
\begin{equation}\label{flow2}
\p_t\g=\frac{1}{\rho v e^w}\left(\s^{ij}-\frac{\g^i\g^j}{v^2}\right)\g_{ij}+\frac{n\sin\varphi\g_\varphi}{v}-\frac{n(\rho^2-1)|\n^\s \g|^2}{2\rho v},\quad \hbox{ in }\ss_+^n\times [0, T),
\end{equation}
with the initial and the boundary conditions 
\begin{eqnarray*}
\g(\cdot, 0)&=&\g_0,\quad   \hbox{ in }\ss^n_+,\\
\nabla^\s_{\p_\varphi}\g&=& 0,\quad \,\, \hbox{ on } \p \ss_+^n\times [0, T).
\end{eqnarray*}
where $\g_0$ is the corresponding function for $x_0$.

\

\section{A priori estimates}
The short time existence of  the scalar flow \eqref{flow2} follows by the standard parabolic PDE theory. Next  we show the $C^0$ and $C^1$ estimates for \eqref{flow2}. The a priori $C^0$ estimate follows directly from the maximum principle.
\begin{prop}\label{C0}
Let $\g: \ss_+^n\times [0, T)\to \rr$ solve \eqref{flow2}. Then 
$$\min_{\ss^n_+} \g_0\le \g\le \max_{\ss^n_+} \g_0.$$
\end{prop}

The key point is the following gradient estimate for $\g$.
\begin{prop}\label{C1}
Let $\g: \ss_+^n\times [0, T)\to \rr$ solve \eqref{flow2}. Then there exists a constant $C$, depending on $\|\g_0\|_{C^1}$ and $\min_{\ss^n_+} \g_0$ such that
$$|\nabla^\s \g|^2\le C.$$
Moreover, if $n\ge 3$, we have 
$$|\nabla^\s \g|^2\le C_1 e^{-C_2 t}.$$
\end{prop}
\begin{proof}For notation simplicity, we use $\nabla=\nabla^\s$ in the proof. Denote
$$F(\nabla^2 \g, \n\g, \rho, \varphi)=\frac{1}{\rho v e^w}\left(\s^{ij}-\frac{\g^i\g^j}{v^2}\right)\g_{ij}+\frac{n\sin\varphi\g_\varphi}{v}-\frac{n(\rho^2-1)|\n \g|^2}{2\rho v},$$
and $$F^{ij}=\frac{\p F}{\p \g_{ij}}, \quad F^{p}=\frac{\p F}{\p \g_p}, \quad F^\rho=\frac{\p F}{\p \rho}, \quad F^\varphi=\frac{\p F}{\p \varphi}.$$
Then 
\begin{eqnarray}\label{grad-eq1}
&&\p_t|\n \g|^2=2\g_k(\g_t)_k=2F^{ij}\g_k\g_{ijk}+ F^p\n_p |\n \g|^2 +2F^\rho \rho|\n\g|^2+2F^\varphi\g_\varphi.
\end{eqnarray}

By a direct computation, we have
\begin{eqnarray}
&&F^{ij}=\frac{1}{\rho v e^w}\left(\s^{ij}-\frac{\g^i\g^j}{v^2}\right),\label{grad-eq2}\\
&&F^\rho=\frac{\rho^2-1}{2\rho^2 v}\left(\s^{ij}-\frac{\g^i\g^j}{v^2}\right)\g_{ij} -\frac{n(\rho^2+1)}{2\rho^2 v}|\n \g|^2,\label{grad-eq3}\\
&&F^\varphi=-\sin\varphi\frac{1}{v}\left(\s^{ij}-\frac{\g^i\g^j}{v^2}\right)\g_{ij} +\frac{n\cos\varphi}{v}\g_\varphi.\label{grad-eq4}
\end{eqnarray}
Using the Ricci identity $$\g_{ijk}=\g_{kij}+\g_j \s_{ki}- \g_k \s_{ij}$$ and \eqref{grad-eq2}, we have
\begin{eqnarray}\label{grad-eq5}
2F^{ij}\g_k\g_{ijk}&=& F^{ij} \n^2_{ij}|\n \g|^2- 2 \frac{1}{\rho v e^w}\left(\s^{ij}-\frac{\g^i\g^j}{v^2}\right)\g_{ik}\g_{jk}-\frac{2(n-1)}{\rho v e^w}|\n\g|^2\nonumber
\\&=& F^{ij} \n^2_{ij}|\n \g|^2-  \frac{2}{\rho v e^w}|\n^2 \g|^2 +\frac{1}{2\rho v^3 e^w}\left|\n|\n \g|^2\right|^2-\frac{2(n-1)}{\rho v e^w}|\n\g|^2.
\end{eqnarray}
Replacing \eqref{grad-eq3}, \eqref{grad-eq4} and \eqref{grad-eq5} into \eqref{grad-eq1}, we get
\begin{eqnarray}\label{eq-grad1}
\p_t|\n \g|^2&= & F^{ij}\n^2_{ij}|\n \g|^2+F^p \n_p |\n \g|^2\nonumber\\&&-\frac{2}{\rho v e^w}|\n^2 \g|^2+\frac{1}{2\rho v^3 e^w}\left|\n|\n \g|^2\right|^2-\frac{2(n-1)}{\rho v e^w}|\n\g|^2\nonumber
\\&&+2 \left[\frac{\rho^2-1}{2\rho^2 v}\left(\s^{ij}-\frac{\g^i\g^j}{v^2}\right)\g_{ij} -\frac{n(\rho^2+1)}{2\rho^2 v}|\n \g|^2\right]\rho|\n\g|^2\nonumber
\\&&+2\left[-\sin\varphi\frac{1}{v}\left(\s^{ij}-\frac{\g^i\g^j}{v^2}\right)\g_{ij} +\frac{n\cos\varphi}{v}\g_\varphi\right]\g_\varphi\nonumber
\\&=& F^{ij}\n^2_{ij}|\n \g|^2+F^p \n_p |\n \g|^2+\left(\sin\varphi-\frac{\rho^2-1}{2\rho }|\n\g|^2\right)\frac{\<\n\g, \n|\n\g|^2\>}{v^3}\nonumber
\\&&-\frac{2}{\rho v e^w}|\n^2 \g|^2+\frac{1}{2\rho v^3 e^w}\left|\n|\n \g|^2\right|^2-\frac{2(n-1)}{\rho v e^w}|\n\g|^2\nonumber
\\&&+\frac{\rho^2-1}{\rho v}\Delta \g |\n\g|^2-\frac{n(\rho^2+1)}{\rho v}|\n\g|^4+\frac{2n\cos\varphi}{v}\g_\varphi^2-\frac{2\sin\varphi}{v}\Delta \g \g_\varphi.
\end{eqnarray}
Now we examine the boundary normal derivative of  $|\n\g|^2$ and have
\begin{eqnarray}\label{bdry}
&&\nabla_{\p_\varphi}|\n\g|^2=2(\g_{\theta_\a}\g_{\theta_\a\varphi}+\g_{\varphi}\g_{\varphi\varphi})
=\g_{\theta_\a}[\n_{\p_{\theta_\a}} (\g_\varphi)-(\n_{\p_{\theta_\a}}\p_\varphi)\g]
=0.
\end{eqnarray}
Here we used $\g_\varphi=0$ along $\p \ss^n_+$ and the fact that $\n_{\p_{\theta_\a}}\p_\varphi=0$.

Assume for $t\in [0,T)$, $\max_{\bar \ss^{n}_+}|\n \g|^2(\cdot, t)=|\n \g|^2(x_t, t)$. If $x_t\in \ss^n_+$, it follows from the maximum point condition that
\begin{eqnarray}\label{critical}
\n|\n\g|^2=0, \quad \n^2|\n\g|^2\le 0.
\end{eqnarray}
If $x_t\in \p \ss^n_+$, we see from \eqref{bdry} that $\n_{\p_{\varphi}}|\n\g|^2=0$, and in turn we also have \eqref{critical}. Thus, for each $t\in [0,T)$, at $x_t$, we have \eqref{critical}.
We choose at $x_t$ local coordinates ${x^1,\cdots x^n}$ such that $\g_1=|\n\g|$. One has $\g_{1i}=0$ for all $i$ by \eqref{critical}. By further rotating the $\{x^2, \cdots, x^n\}$ coordinate, we can assume $\n^2\g$ is diagonal.
Then 
$$|\n^2 \g|^2\ge \frac{1}{n-1}(\Delta \g)^2.$$
It follows from \eqref{eq-grad1} that at $x_t$,
\begin{eqnarray}\label{eq-grad''}
0\le \p_t |\n \g|^2(x_t, t)\nonumber
&\le & -\frac{2}{\rho v e^w}|\n^2 \g|^2-\frac{2(n-1)}{\rho v e^w}|\n\g|^2\nonumber
\\&&+\frac{\rho^2-1}{\rho v}\Delta \g |\n\g|^2-\frac{n(\rho^2+1)}{\rho v}|\n\g|^4+\frac{2n\cos\varphi}{v}\g_\varphi^2-\frac{2\sin\varphi}{v}\Delta \g \g_\varphi\nonumber
\\ &\le&  -\frac{2(1-\epsilon)}{(n-1) \rho v e^w}\left(\Delta \g- \frac{(n-1)(\rho^2-1)e^w}{4(1-\epsilon)}|\n\g|^2\right)^2\nonumber\\&&
-\frac{2\epsilon}{(n-1) \rho v e^w}\left(\Delta \g+\frac{(n-1)\rho e^w\sin\varphi}{2\epsilon} \g_\varphi\right)^2\nonumber
\\&&+\frac{1}{\rho v}\left(\frac{(n-1)(\rho^2-1)^2e^w}{8(1-\epsilon)}-n(\rho^2+1)\right)|\n \g|^4\nonumber\\&&+
\frac{1}{v}\left(-\frac{2(n-1)}{\rho e^w}|\n\g|^2+2n\cos\varphi\g_\varphi^2+\frac{(n-1)\rho e^w\sin^2\varphi }{2\epsilon} \g_\varphi^2\right).
\end{eqnarray}
Choosing $\epsilon=\frac34$, we have
\begin{eqnarray*}
&&\frac{(n-1)(\rho^2-1)^2e^w}{8(1-\epsilon)}-n(\rho^2+1)\\&<&\frac{ne^w}{2}[(\rho^2-1)^2- (\rho^2+1)(1+\rho^2+2\rho\cos\varphi)]\le - n\rho^2 e^w\end{eqnarray*}
and
\begin{eqnarray*}
&&-\frac{2(n-1)}{\rho e^w}|\n\g|^2+2n\cos\varphi\g_\varphi^2+\frac{(n-1)\rho e^w\sin^2\varphi }{2\epsilon}\g_\varphi^2\\
&\le &\left(-\frac{(n-1)(1+\rho^2+2\rho\cos\varphi)}{\rho}+2n\cos\varphi + \frac{4(n-1)}{3}\frac{\rho}{1+\rho^2+2\rho\cos\varphi}\right)|\n \g|^2\\
&\le &(-2(n-1)+2\cos \varphi+ \frac{2(n-1)}{3})|\n\g|^2
\\&\le&(-\frac43 n+\frac{10}{3})|\n\g|^2. \end{eqnarray*}
Thus
\begin{eqnarray}\label{eq-grad}
0\le \p_t |\n \g|^2 &\le&   - \frac{n\rho e^w}{v}|\n \g|^4 +(-\frac43 n+\frac{10}{3})\frac{1}{\rho v}|\n\g|^2.
\end{eqnarray}
It follows from \eqref{eq-grad} that $|\n\g|^2\le C$. Moreover, when $n\ge 3$, one sees from \eqref{eq-grad} that $|\n\g|^2\le C_1e^{-C_2t}$.
\end{proof}

\section{Global convergence}

We first prove the nice properties of \eqref{flow}, mentioned in the Introduction.

\begin{prop}\label{pro3.3}\label{first-var} Flow \eqref{flow} satisfies
\begin{eqnarray}\frac d {dt}{\rm Vol}(\Omega_t) =0
\end{eqnarray}
and
\begin{eqnarray}
&&\frac{d}{dt}{\rm Area}(\S_t)=-\frac{1}{n-1}\int_{\S} \sum_{i<j} (\kappa_i-\kappa_j)^2\<X_{n+1},\nu\> dA_t \le 0.
\end{eqnarray}
\end{prop}
\begin{proof}From  \eqref{eq3}, we get
$$\frac d{dt}{\rm Vol}(\Omega_t) =\int _\Sigma (n x_{n+1}- H\<X_{n+1}, \nu\>)dA_t=0. $$
The first variational formula gives
\begin{eqnarray*}
\frac{d}{dt}{\rm  Area}(\S_t)= \int_{\S} H(nx_{n+1}-H\<X_{n+1},\nu\>)dA_t.
\end{eqnarray*}
Using  the Minkowski formula \eqref{eq4}
\begin{eqnarray*}
 \int_{\S} Hx_{n+1}-\frac{2}{n-1}\s_2(\kappa)\<X_{n+1},\nu\>dA_t=0,
\end{eqnarray*}
we get
\begin{eqnarray*}
\frac{d}{dt}{\rm  Area}(\S_t)&=& -\int_{\S} \left(H^2- \frac{2n}{n-1}\s_2(\kappa)\right)\<X_{n+1},\nu\> dA_t\\
&=&-\frac{1}{n-1}\int_{\S} \sum_{i<j} (\kappa_i-\kappa_j)^2\<X_{n+1},\nu\>)dA_t \le 0.
\end{eqnarray*}

\end{proof}

Now we prove the global convergence.

\

\noindent{\it Proof of Theorem \ref{mainthm}.}
In view of Proposition \ref{corresp} (iii), the $C^0$ and $C^1$ estimates in Propositions \ref{C0} and \ref{C1} imply that $ \<X_{n+1},\nu\>\ge c>0$, that is, the star-shapedness of $\S_t$ is preserved under the flow \eqref{flow}.

Now we are ready to prove the long time existence in Theorem \ref{mainthm}. Since equation \eqref{flow2} is a quasilinear parabolic PDE of divergent form, the higher order a priori estimates follows from the standard 
parabolic PDE theory, once we have the $C^0$ and $C^1$ estimates in Propositions \ref{C0} and \ref{C1}. Hence we prove that \eqref{flow2} has a smooth solution for all time.
The exponential convergence for $n\ge 3$ follows directly from Proposition \ref{C1}.

For the convergence part in two dimensions, we examine the monotonicity of the area functional along the flow.
In the following we restrict to $n=2$. By integrating \eqref{first-var} over $t\in [0,\infty)$ and using the uniform estimate, we get
\begin{eqnarray*}
\int_0^\infty\int_{\ss^n_+}  |\kappa_1(y, t) -\kappa_2(y, t) |^2 \<X_{n+1},\nu\> dA_t dt\le C.
\end{eqnarray*}
where $\kappa_i(y,t)$, $i=1,2$  are the principal curvatures of the radial graph at $(y,t)$.
It follows from the uniform bound for $ \<X_{n+1},\nu\>$ and  $dA_t$ that 
\begin{equation}\label{k}
\max_{y\in\bar \ss^n_+}|\kappa_1-\kappa_2|(y, t)=o_t(1),\end{equation}
where $o_t(1)$ denotes a  quantity which goes to zero as $t\to \infty$.
See the proof of Proposition 5.5 in \cite{GL}. With the help of the property \eqref{k}, we can show the smooth convergence of flow \eqref{flow} when $n=2$. This idea was used first by Guan-Li in \cite{GL}.

Let us go back to the estimate at $x_t$, where $\max_{\bar \ss^{n}_+}|\n \g|^2(\cdot, t)=|\n \g|^2(x_t, t)$. Again we choose the local coordinate around $x_t$ such that at $x_t$, $$\g_1=|\n\g|, \quad \g_{11}=0.$$
In view of Proposition \ref{corresp} (iv), the Weingarten transformation $h_i^j$ is diagonal in this coordinate which means the coordinate directions are the principal directions of $x(\cdot, t)$ at $x_t$. Thus the principal curvature $\kappa_i$ at $x_t$
 is given by
$$\kappa_i=\frac{\g_{ii}}{\rho v e^w}+\frac{\sin\varphi\g_\varphi}{v}+\frac{(\rho^2-1)}{2\rho v},\quad  i=1, 2.$$
It follows that at $x_t$, \begin{eqnarray}\label{add3}
|\Delta \g|=|\g_{22}+\g_{11}|=|\g_{22}-\g_{11}|= \rho ve^w|\kappa_2-\kappa_1|=o_t(1).
\end{eqnarray}
Using \eqref{add3} and the $C^1$ estimate, we get at $(x_t, t)$,
\begin{eqnarray}
\p_t|\n \g|^2&\le & -\frac{2}{\rho v e^w}|\n^2 \g|^2-\frac{2(n-1)}{\rho v e^w}|\n\g|^2\nonumber
\\&&+\frac{\rho^2-1}{\rho v}\Delta \g |\n\g|^2-\frac{n(\rho^2+1)}{\rho v}|\n\g|^4+\frac{2n\cos\varphi}{v}\g_\varphi^2-\frac{2\sin\varphi}{v}\Delta \g \g_\varphi\nonumber
\\ &\le& -\frac{n(\rho^2+1)}{\rho v}|\n\g|^4+ \frac{1}{v}\left(-\frac{2}{\rho e^w}|\n\g|^2+4\cos\varphi\g_\varphi^2\right)+o_t(1)\nonumber
\\&\le &-C|\n \g|^4+o_t(1).\label{add1}
\end{eqnarray}
Here we have used  $$-\frac{2}{\rho e^w}|\n\g|^2+4\cos\varphi\g_\varphi^2\le \left(-\frac{1+\rho^2+2\rho\cos\varphi}{\rho}+ 4\cos \varphi\right) |\n\g|^2\le 0.$$
Now we claim that $$|\n\g|^2=o_t(1).$$  The smooth convergence follows from this claim and the interpolation theorem.  We show the claim in two steps.

First, we show that there exists a sequence $\{t_i\}$ with $t_i\to \infty$ such that $$\max_{\bar\ss^n_+}|\nabla \gamma (\cdot, t_i)|^2 \to 0\quad \hbox{ as }i\to \infty.$$ 
Assume this is not true. Then there exists $\e_0>0$ and $T_0>0$ such that  $$\max_{\bar\ss^n_+}|\nabla \gamma (\cdot, t)|^2 \ge \e_0, \quad\hbox{ for }t>T_0.$$ From \eqref{add1} we have that for a large $T_1>0$ and for any $t>T_1$,
 we have
$$ \frac{d}{dt}\max_{\bar\ss^n_+}| \n \gamma|^2 \le -C \max_{\bar\ss^n_+}| \n \gamma|^4+ \frac 12 C\e_0^4 = -\frac 12 C \e_0^4,
$$
which is impossible.

Second,  we show that for any  sequence $\{s_i\}$ with $s_i\to \infty$, we have $$\max_{\bar\ss^n_+}|\nabla \gamma (\cdot, s_i)|^2 \to 0\quad \hbox{ as }i\to \infty.$$  If not, 
there exists a sequence $\{s_i\}$ with $s_i\to \infty$ such that $$\max_{\bar\ss^n_+}|\nabla \gamma (\cdot, s_i)|^2  \ge \e_1$$ for any $s_i$ and for some positive constant $\e_1$.
Without loss of generality, we may assume that $t_i<s_i$. We consider the interval $I_i:=[t_i, s_i]$ for sufficiently large $i$, such that we have from \eqref{add1} at a maximum point $x_t\in \bar\ss^n_+$
\begin{equation}\label{add2} \frac{d}{dt}\max_{\bar\ss^n_+}| \n \gamma|^2 \le  -C\max_{\bar\ss^n_+}| \n \gamma|^4+  \frac 12 C \e_1^4 \end{equation}
for any $t\ge t_i$. Let $y_i \in  \bar\ss^n_+$  and $ \bar t_i \in [t_i, s_i]$ such that $$|\nabla \gamma (y_i, \bar t_i)|^2=\max_{t\in [t_i,s_i]}\max_{\bar\ss^n_+}|\nabla \gamma (\cdot, t)|^2\ge \e_1.$$
By the first step, we may assume that $\bar t_i\not =t_i$ for $i$ large. It follows that $$\frac{d}{dt}\max_{\bar\ss^n_+}| \n \gamma|^2(\bar t_i)\ge 0.$$ 
Together with \eqref{add2}, implies that $$|\nabla \gamma (y_i, \bar t_i)|^2 <\e_1,$$ a contradiction. This proves the claim.

From the claim, it follows easily that $\gamma(t)$ converges smoothly to a constant $\gamma_0$
and  $\rho\to\rho_0$ smoothly for some constant $\rho_0>0$, depending on the initial enclosed volume of $x_0$. 

Next we show the exponential convergence in the case $n=2$ and the enclosed volume of $x_0$ is not that of a half ball. In this case, $\rho_0\not = 1$. We return to \eqref{eq-grad''}. By choosing $\epsilon<1$ close to $1$, we have
\begin{eqnarray*}
 \p_t |\n \g|^2(x_t, t)&\le& \frac{1}{\rho v}\left(\frac{(\rho^2-1)^2e^w}{8(1-\epsilon)}-n(\rho^2+1)\right)|\n \g|^4\nonumber\\&&+
\frac{1}{v}\left(-\frac{2}{\rho e^w}+4\cos\varphi +\frac{\rho e^w\sin^2\varphi }{2\epsilon} \right)|\nabla \g|^2
\\&\le & \frac{1}{\rho v}\left(\frac{(\rho^2-1)^2e^w}{8(1-\epsilon)}-n(\rho^2+1)\right)|\n \g|^4\nonumber\\&&+
\frac{1}{v}\left(-\frac{(1-\rho\cos\varphi)^2}{\rho}+\rho\sin^2\varphi \left(\frac{1}{\epsilon(1+\rho^2+2\rho\cos\varphi)}-1 \right)\right)|\nabla \g|^2
\\&\le& C|\n \g|^4-\left(\frac{(1-\rho\cos\varphi)^2}{\rho}+C\rho\sin^2\varphi\right)|\nabla \g|^2.
\end{eqnarray*}
As $\rho$  converges to $\rho_0\not = 1$, 
$$-\left(\frac{(1-\rho\cos\varphi)^2}{\rho}+C\rho\sin^2\varphi\right)\le -C_1$$ for some $C_1>0$ and $t$ large. Then the exponential convergence follows. 

\qed

\

\noindent{\it Proof of Corollary \ref{coro1}.} It follows from Theorem \ref{mainthm} and Proposition \ref{pro3.3}. \qed

\

\end{document}